\documentclass[10pt]{article}
\usepackage{amsfonts,amsmath}
\usepackage{amsfonts}
\usepackage{amsmath}   
\usepackage{amsthm}    
\usepackage{amscd}     
\usepackage[utf8]{inputenc}

\usepackage{euscript}

\def\C{{\mathbb C}}
\def\P{{\mathbb P}}

\usepackage[T1]{fontenc}

\usepackage{amsfonts}
\usepackage{amsmath}
\usepackage{amscd} 
\usepackage{graphicx}

\newtheorem{defi}{D\'{e}finition}[section]

\newtheorem{thm}[defi]{Theorem}

\def\P{\mathbb{P}}

\def\O{{\mathcal O}}

\def\eqref#1{(\ref{#1})}
\def\:{\colon }

\swapnumbers

\title{A vector bundle proof of Poncelet theorem}
\author{Jean Vallès \footnote{Author partially supported by ANR-09-JCJC-0097-0 INTERLOW
and ANR GEOLMI}}

\begin{document}

\maketitle

\begin{abstract}

In the town  of Saratov where he was prisonner, Poncelet, continuing the work of Euler and Steiner on 
polygons simultaneously inscribed in a circle and circumscribed around an other circle, proved the following generalization. 

\smallskip

\textbf{Theorem}
\textit{ Let $C$ and $D$  be two  smooth conics in $\P^2(\C)$. If 
 $D$ passes through the $\binom{n}{2}$ vertices of a complete polygon with $n$ sides tangent to $C$ then 
$D$ passes through the vertices of infinitely many such polygons.}

\smallskip

According to Berger \cite{Be} this theorem is the nicest result about the geometry of conics. Even if it is,  there are few proofs of it.
To my knowledge there are only three.
The first proof, published in 1822 and  based on infinitesimal deformations, is due to Poncelet (\cite{Po}).
Later, Jacobi proposed a new proof based on finite order points on elliptic curves; his proof, certainly the most famous, is  
explained in a modern way and  in detail by Griffiths and Harris  (\cite{GH}).  In 1870 Weyr proved a Poncelet theorem 
in space (more precisely for two quadrics) that implies  the one above when one quadric is a cone; this proof is 
explained by Barth and Bauer (\cite{BB}). 

\medskip

 Our aim in this short note is to involve vector bundles techniques to propose a 
new proof of this celebrated result. Poncelet did not appreciate Jacobi's for the reason that  it was too far 
from the geometric intuition. I guess that he would not appreciate our proof either   for the same reason.

\end{abstract}
\section{Preliminaries}
\label{prelim}
In all this text the ground field is $\C$. A set of    
$n$ vertices connected by $n$ distinct lines form a $n$-gon, when a set consisting of $n$ distinct lines with  their  $\binom{n}{2}$ vertices  form a complete $n$-gon.

 \begin{figure}[h!]
    \centering
    \includegraphics[height=6.5cm]{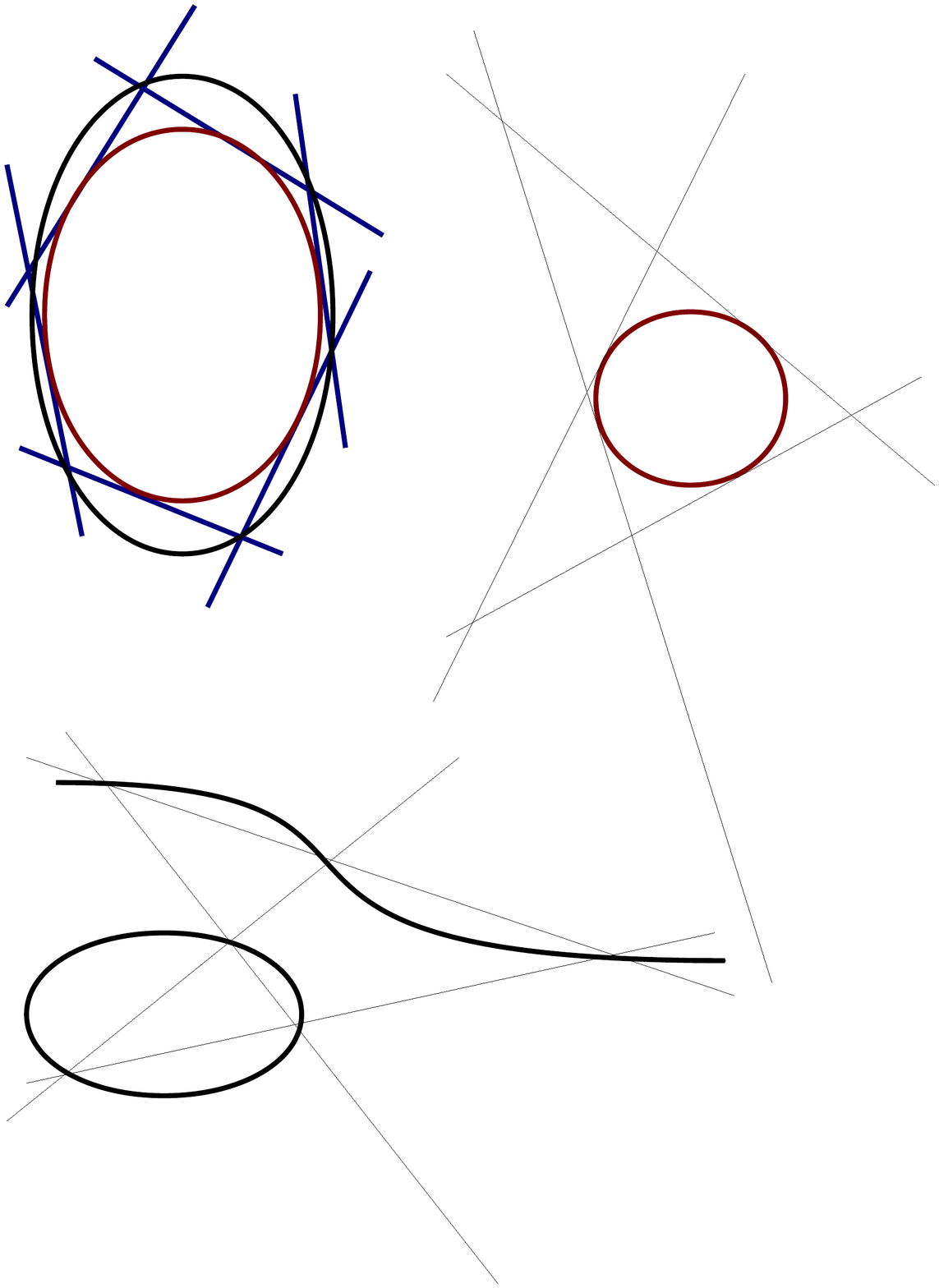}
    \caption{A complete $5$-gon inscribed in a cubic, a complete $4$-gon circumscribed around a conic and a pentagon simultaneously inscribed 
in a conic and circumscribed around another}
  \end{figure}

We say that a $n$-gon (respectively a complete $n$-gon) is inscribed in a given curve if this curve passes through the  $n$-vertices (respectively the   $\binom{n}{2}$ vertices).
We say that a $n$-gon, or a complete $n$-gon, is circumscribed around a smooth conic $C$ if the sides of the polygon, i.e. the $n$ lines, are tangent to the $C$.

\section{Schwarzenberger bundles}
First of all let us introduce a vector bundle $E_{n,C}$  naturally associated to 
any set of $n$ lines tangent to a fixed smooth conic $C\subset \P^2$  (they were defined by Schwarzenberger in \cite{S}). A set of $n$ lines tangent 
to $C$ corresponds by projective duality to a degree $n$ divisor on the dual conic $C^{\vee}\subset \P^{2\vee}$.
According to the isomorphism $C^{\vee}\simeq \P^1$ we can consider the subvariety 
 $X \subset \P^2 \times \P^{1}$ defined by the equation $x_0u^2+x_1uv+x_2v^2=0$ ($(x_0,x_1,x_2)$ are the homogeneous 
coordinates on $\P^2$ and $(u,v)$ the homogeneous coordinates on $\P^1$) 
 and the
projections $p$ and $q$ respectively on $\P^2$ and $\P^1$.
$$ \begin{CD}
 X @>q>> \P^{1} \\ @VpVV \\ \P^{2}
\end{CD}
$$  
The variety $X$ is  a double cover of $\P^2$ ramified along $C$. If  $x\in \P^1$ then $p(q^{-1}(x))$ is a line
in $\P^2$ tangent to $C$. If, instead of considering a point, we are considering a line bundle on $\P^1$
we will find a vector bundle of rank two on $\P^2$ by taking the  direct image of  its inverse image.
Moreover, following Scwharzenberger, we know a very explicit resolution of this bundle. Indeed, tensorizing the following exact sequence 
$$ 
\begin{CD} 0 @>>> \O_{\P^2\times \P^1}(-1,-2) @>>>\O_{\P^2\times \P^1} @>>> \O_{X} @>>>0, 
\end{CD}
$$
by $q^{*}\O_{\P^1}(n)$ and taking its direct image by $p$ we have :
$$ 
 0 \longrightarrow  \mathrm{H}^0(\O_{\P^1}(n-2))\otimes \O_{\P^2}(-1) \stackrel{M}\longrightarrow   \mathrm{H}^0(\O_{\P^1}(n))\otimes \O_{\P^2} \longrightarrow  E_{n,C} 
\longrightarrow 0.
$$
The map $M$ can be represented by the matrix of linear forms : 
$$ M= \left (
             \begin{array}{ccccc}
              x_0 &     &         &    \\
              x_1 & x_0 &         &    \\
              x_2 & x_1 & \ddots  &     \\
                  & x_2 & \ddots  & x_0 \\
                  &     & \ddots  & x_1 \\
                  &     &         & x_2 
             \end{array}
      \right )
$$
Let us show that 
the zero locus $Z(s)$ of a non zero  section $s\in \mathrm{H}^0(E_{n,C})$  is the set of $\binom{n}{2}$ vertices of the 
  $n$ tangent lines to $C$ given by the corresponding $n$ points on $C^{\vee}$. We denote by $D_n$ this set of $n$ points on $C^{\vee}$.
Since $\mathrm{H}^0(\O_{\P^1}(n))=\mathrm{H}^0(E_{n,C})$, the section $s$
corresponds to an
 hyperplane $H_s \subset \P(\mathrm{H}^0(\O_{\P^1}(n)))$. This hyperplane meets  the image $v_n(\P^1)$ of $\P^1\simeq C^{\vee}$ in $ \P(\mathrm{H}^0(\O_{\P^1}(n)))$ 
(by the Veronese imbedding $v_n$)
 along $n$ points which correspond to the points  of the divisor $D_n$. The section $s$
induces a rational map
$
\pi_s : \P^2 \longrightarrow \P((E_{n,C})^{\vee})
$
which is not defined over the zero-scheme $Z(s)$. More precisely let $x$ be a point in $\P^2$ and  $L_x\subset \P^{2\vee}$ its dual line. This dual line
corresponds by the Veronese morphism  to a two-secant line of $v_n(\P^1)$ (call it $L_x$ again). 
If  $L_x$ is not a two secant line to $D_n$ there is exactly one intersection point 
$L_x\cap H_s$ which is the image of $x$ by $\pi_s$. Conversely the map $\pi_s$ is not well defined when  $L_x\subset H_s$, i.e. when 
  $L_x$ is  a two-secant line to $D_n$, or equivalently when   $x$  is a vertex of two tangent lines to $C$ along 
$D_n$.

\section{Darboux theorem}
We can prove now the so-called darboux theorem (\cite{Da}, page 248).
\begin{thm}
\label{da}
Let  $S\subset \P^2$ be a  curve of degree $(n-1)$. If there is a 
complete $n$-gon (polygon with $n$ sides)  tangent to a smooth conic  $C$ and inscribed into $S$, then there 
are  infinitely many of them.
\end{thm}
\begin{proof}
I recall here a proof already written in \cite{Va}.
A complete $n$-gon circumscribed around $C$ and inscribed into $S$ corresponds to a  non-zero global section 
$s\in  \mathrm{H}^0(E_{n,C})$ vanishing along its vertices $Z(s)$:
$$ 
\begin{CD} 0 @>>> \O_{\P^2} @>>>E_{n,C} @>>>\mathcal{I}_{Z(s)}(n-1) @>>>0.
\end{CD}
$$
By the remark of the previous section (\ref{prelim}), the curve $S$ corresponds to a global section of  $\mathcal{I}_{Z(s)}(n-1)$.
Since the map 
$$\begin{CD} 
 \mathrm{H}^0(E_{n,C})@>>> \mathrm{H}^0(\mathcal{I}_{Z(s)}(n-1))  
  \end{CD}
$$ is surjective, there exists  a non-zero section
$t\in  \mathrm{H}^0(E_{n,C})$ (i.e. another $n$-gon)
such that the determinant
$$\begin{CD} 
 \O_{\P^2}^2@>(s,t)>> E_{n,C}  
  \end{CD}
$$
 is the equation of $S$. This proves the theorem since any linear combination of $s$ and $t$  vanish along the vertices of a 
complete $n$-gon.
\end{proof}
 \begin{figure}[h!]
    \centering
    \includegraphics[height=6.5cm]{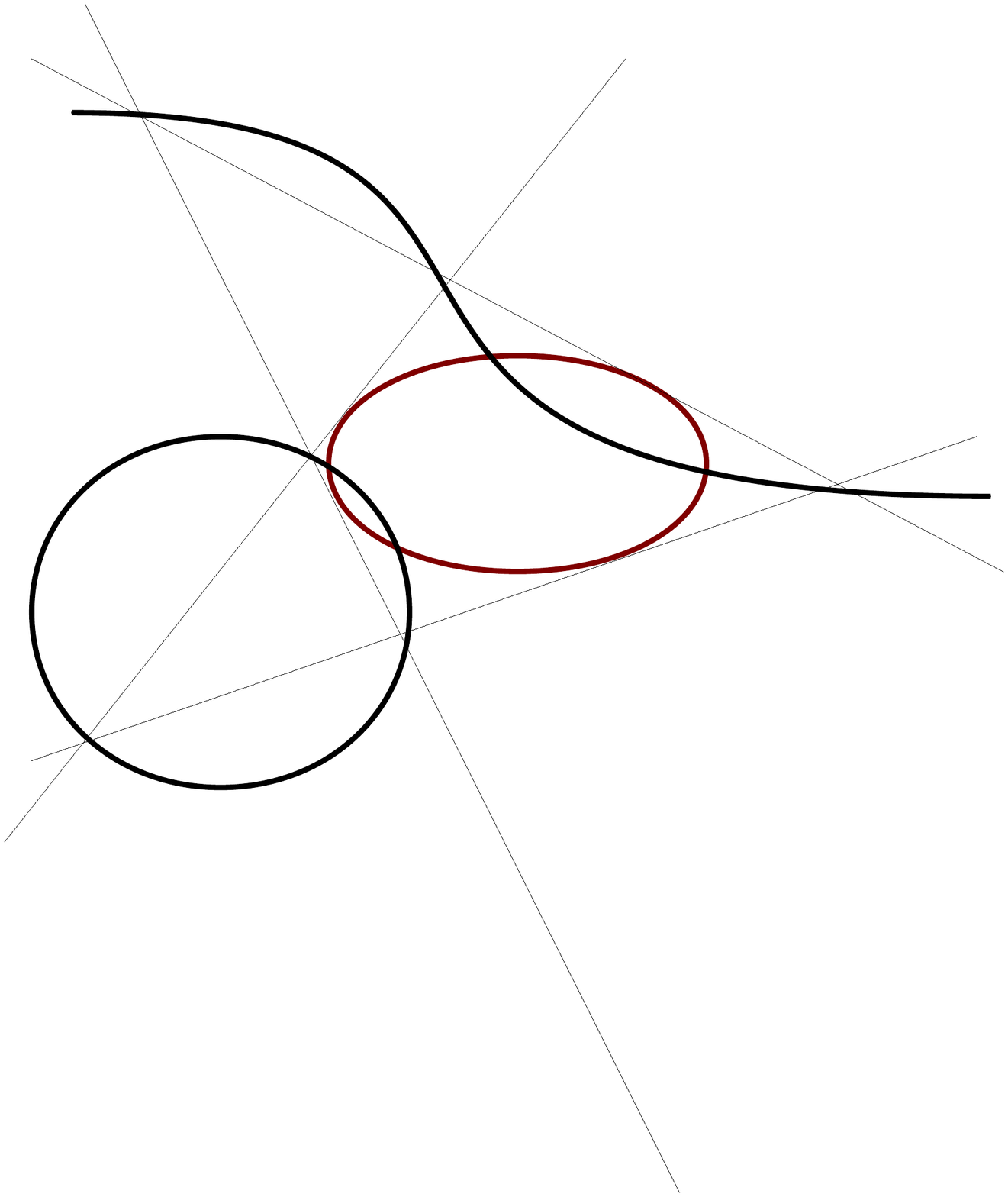}
    \caption{A complete circumscribed 4 gon and a cubic Poncelet curve}
  \end{figure}
These curves described by Darboux are called \textit{Poncelet curves}. When $n=5$ they are the so-called Luröth quartics (see \cite{OS}).
\section{Poncelet theorem}
Let us now consider $n$-gons   that are simultaneously inscribed in a smooth conic and circumscribed around a other one.
 For these configurations Poncelet proved (\cite{Po}, page 362) :

\begin{thm}
Let  $C\subset \P^2$ and $D\subset \P^2$ be two smooth conics such that there exist one $n$-gon inscribed in $D$ and circumscribed around $C$.
Then there are infinitely many of such $n$-gons.
\end{thm}
 \begin{figure}[h!]
    \centering
    \includegraphics[height=6.5cm]{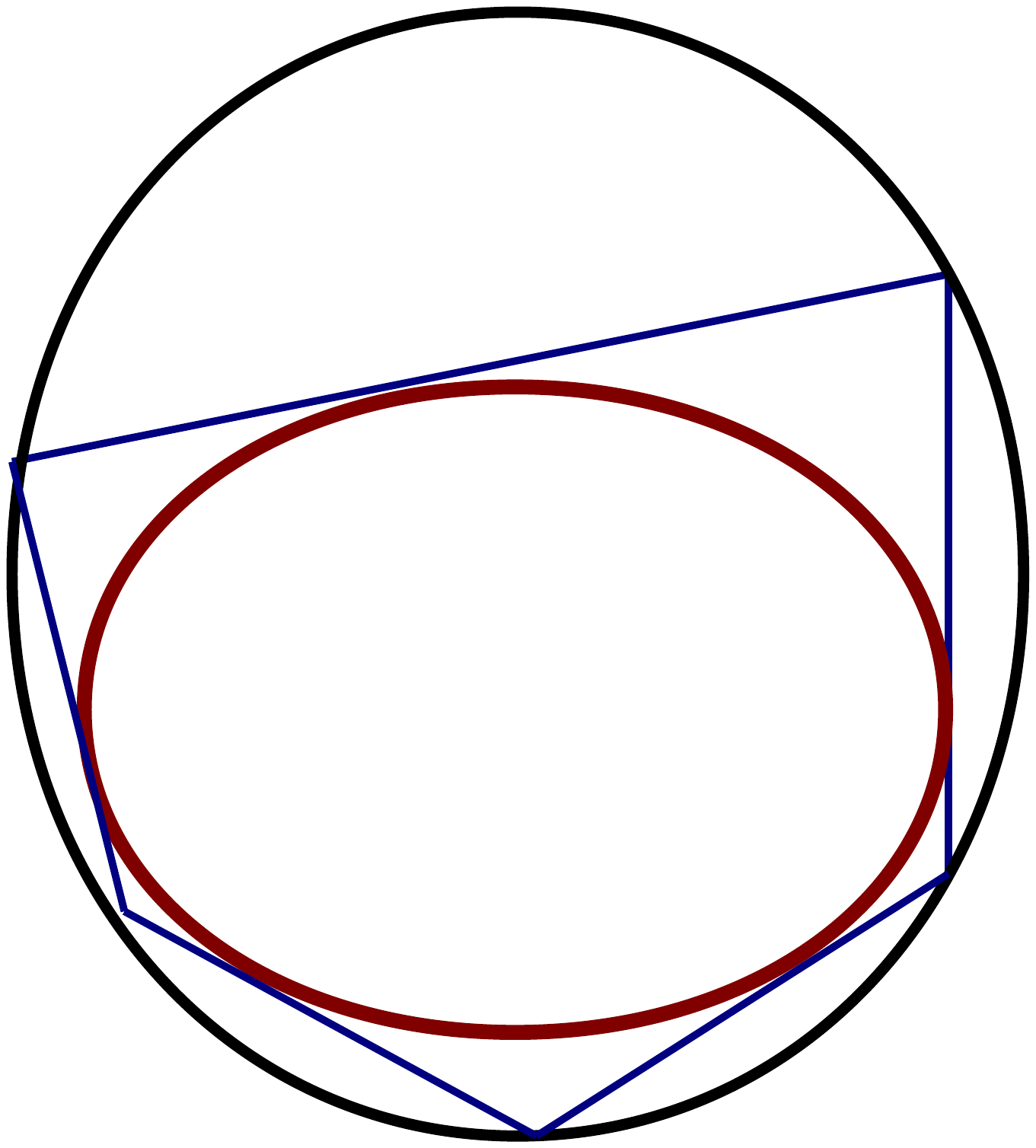}
    \caption{pentagon inscribed and circumscribed }
  \end{figure}
\begin{proof}
Let us consider one such $n$-gon. It is the union of $n$ lines
 $l_1, \cdots, l_n$ with the $\binom{n}{2}$ vertices $l_i\cap l_j$ for $1\le i,j \le n$ and $i\neq j$. 
There is a non-zero section $s\in \mathrm{H}^0(E_{n,C})$ vanishing along the vertices of these lines. We denote by $Z(s)$ the set of these vertices.

\smallskip

Let us tensorize the following exact sequence  
$$ 
\begin{CD} 0 @>>> \O_{\P^2} @>>>E_{n,C} @>>>\mathcal{I}_{Z(s)}(n-1) @>>>0
\end{CD}
$$
by $\O_D$. Since  $D\cap Z(s)$ consists in $n$-points, it induces the following decomposition of $E_{n,C}$ along $D$ : 
$$ E_{n,C}\otimes \O_D= \O_D(\frac{n-2}{2})\oplus \O_{D}(\frac{n}{2}).$$ 
According to this decomposition, we  consider the following exact sequence : 
 $$\begin{CD}
0@>>> F @>>> E_{n,C} @>>> \O_D(\frac{n-2}{2}) @>>>0  
  \end{CD}
$$
where  $F$ is a rank two vector bundle over  $\P^2$. Taking the cohomology long exact sequence we verify immediately 
that  $h^0(F)\ge 2$. Then, let us consider a pencil of sections of  $F$ and also the pencil of sections of  $E_{n,C}$ induced by it.
We obtain a commutative diagram:
$$\begin{CD}
@. \O_{\P^2}^2 @= \O_{\P^2}^2 @. @.\\ 
@. @VVV  @VVV @.\\  
0@>>> F @>>> E_{n,C} @>>> \O_D(\frac{n-2}{2}) @>>>0 \\
@. @VVV  @VVV @|\\
0@>>> \mathcal{L}_1 @>>> \mathcal{L}_2 @>>> \O_D(\frac{n-2}{2}) @>>>0.
  \end{CD}
$$
The sheaf  $\mathcal{L}_2$ is supported by a  curve  $\Gamma_2$ of degree $(n-1)$ that is the determinant of a pencil of sections 
of $E_{n,C}$. This curve $\Gamma_2$ is a  Poncelet curve. Then a general  point on $\Gamma_2$ is a vertex of a complete 
$n$-gon inscribed in $\Gamma_2$ and circumscribed around $C$.  Moreover any intersection point of the $n$ lines forming the $n$-gon
with $\Gamma_2$ is a vertex of this $n$-gon (it is clear by B\'ezout theorem since $n(n-1)=2\times \binom{n}{2}$).
Let $\Gamma_1$ be the curve supporting the sheaf $\mathcal{L}_1$. We have of course  $\Gamma_2=D\cup \Gamma_1$.
  Then  $D$ is an irreducible component of a Poncelet curve and by the way any general  point on $D$ is the vertex of complete $n$-gon inscribed in $\Gamma_2$.
Then this configuration meets the conic $D$ in at least  (because there are $n$ lines) and at most  (because they are vertices and the decompostion 
of the bundle along $D$ is fixed) $n$-points, so exactly
 $n$-points, each counting doubly.

\end{proof}

\end{document}